\documentclass{amsart}
\usepackage{amsmath}
\usepackage{amsthm}
\usepackage[applemac]{inputenc} 
\usepackage{hyperref}

\usepackage{amssymb}
\usepackage{graphicx}
\usepackage{enumitem}
\usepackage{amsfonts}
\usepackage{array}
\usepackage{graphicx}
\usepackage{xcolor}
\usepackage{ytableau}
\ytableausetup{centertableaux}
\newcommand{\f}{\varphi}
\newcommand{\G}{\mathbb G}

\newcommand{\C}{\mathbb C}
\newcommand{\PP}{\mathbb P}

\newcommand{\cV}{\mathcal V}

\newcommand{\be}{\mathbf e}
\newcommand{\la}{\lambda}

\newcommand\edps{md-pairs\ }
\newcommand\edp{md-pair\ }

\DeclareMathOperator{\ed}{e.\!d.}

\DeclareMathOperator{\codim}{codim}

\DeclareMathOperator{\DA}{A}

\DeclareMathOperator{\GL}{GL}

\newcolumntype{C}{>{$}c<{$}}

\newtheorem{theorem}{Theorem}[section]
\newtheorem{definition}[theorem]{Definition}

\newtheorem{lemma}[theorem]{Lemma}
\newtheorem{proposition}[theorem]{Proposition}
\newtheorem{corollary}[theorem]{Corollary}

\newtheorem{example}[theorem]{Example}

\title{Morphisms between Grassmannians, II}
\author{Gianluca Occhetta}
\address{Dipartimento di Matematica - Università di Trento, via
Sommarive 14 I-38123  Trento (TN), Italy}
\email{gianluca.occhetta@unitn.it}
\author{Eugenia Tondelli}
\address{Dipartimento di Matematica - Università di Trento, via
Sommarive 14 I-38123  Trento (TN), Italy}
\email{eugenia.tondelli@gmail.com}
\begin{document}

\begin{abstract} Denote by $\G(k,n)$ the Grassmannian of linear subspaces of dimension $k$ in $\PP^n$. We show that if $\f:\G(l,n) \to \G(k,n)$ is a nonconstant morphism and $l \not=0,n-1$ then $l=k$ or $l=n-k-1$ and $\f$ is an isomorphism. 
\end{abstract}
\maketitle
\section{Introduction}

In \cite{Tango1}, Tango proved that there are no nonconstant morphisms from $\PP^m$ to the Grassmannian $\G(k,n)$ if $m >n$, and later, in \cite{Tango2}, considered the case $m=n$, proving the same result for $kn$ even, $(k,n) \not = (2,5)$, $k \not \in \{0,n-1\}$. The result in \cite{Tango1} has been generalized in \cite{NO22} to the case of morphisms $\f:\G(l,m) \to \G(k,n)$, with $m >n$, and later to a more general setting (see Theorem \ref{teo:mos8} and references therein).

 The aim of the present paper is to generalize the results in \cite{Tango2}, considering morphisms $\G(l,n) \to \G(k,n)$, and proving the following:

\begin{theorem}\label{thm:main} If $\f:\G(l,n) \to \G(k,n)$ is a nonconstant morphism and $l \not=0,n-1$ then $l=k$ or $l=n-k-1$ and $\f$ is an isomorphism.
\end{theorem}

In a nutshell, the idea  in \cite{Tango1} and \cite{NO22} to prove the constancy of a morphism $\f:M \to \G(k,n)$ was to consider the relation among Chern classes coming from the universal sequence on $\G(k,n)$, pull it back via $\f$, and show that this leads to a contradiction, via a study of the Chow ring $\DA^\bullet(M)$.

In \cite{MOS8} this idea was refined and reinterpreted geometrically, considering the Schubert varieties $X_H, X_p \subset \G(k,n)$,  parametrizing linear spaces contained in a hyperplane $H$ or passing through a point $p$; then $[X_H] \in \DA^{k+1}(\G(k,n))$, $[X_p] \in \DA^{n-k}(\G(k,n))$  and clearly $[X_H] \cdot [X_p]=0$.  If the pullback via $\f$ of one of the two cycles is zero one can construct   a morphism from $M$ to a smaller Grassmannian which factors via $\f$, allowing inductive arguments. Else, one obtains two effective nonzero cycles  whose product is zero in $\DA^{n+1}(M)$. The proof can then be finished by showing that there are no such pairs in $\DA^\bullet(M)$. This is the idea that led to the notion of effective good divisibility of a variety (see Section \ref{ssec:egd}).

The last step of the above argument could be further refined: in fact it is enough to show that for every effective nonzero $x \in \DA^{k+1}(M), y\in \DA^{n-k}(M)$ it holds $x \cdot y \not =0$. We fulfill this task, in our setup, by characterizing the pairs of effective nonzero cycles of total codimension $n+1$ in $\DA^\bullet(\G(l,n))$ whose product is zero, called maximal disjoint pairs (Corollary \ref{cor:mdp}). In particular, for $\G(l,n)$ such a pair consists of cycles of codimensions $l+1$ and $n-l$, forcing $l=k$ or $n-k-1$ in order for $\f:\G(l,n) \to \G(k,n)$ to be nonconstant. In these cases we conclude that $\f$ is an isomorphism using a Remmert-Van de Ven type theorem for rational homogeneous varieties due to Hwang and Mok (\cite[Main Theorem]{HM}).

\section{Preliminaries}

\subsection{Effective Good Divisibility}\label{ssec:egd}

The notion of effective good divisibility of a smooth complex projective variety $M$ (see \cite[Section 2.1]{MOS8}) is related to the total codimension of effective zero divisors in the Chow ring $\DA^{\bullet}(M)$.

\begin{definition}
The {\em effective good divisibility} of $M$, denoted by $\ed(M)$, is the maximum integer $s$ such that, given effective cycles $x_i \in \DA^{i}(M),$  $x_j \in \DA^{j}(M)$ with $i+j  \le s$ and $x_i  x_j=0$, then either $x_i = 0$ or $x_j = 0$.
\end{definition}

\begin{example} The effective good divisibility is known for the wide class of rational homogeneous manifolds:  it was computed for Grassmannians in \cite{NO22}, for varieties of classical type independently in \cite{MOS8} and \cite{HLL}, for varieties of exceptional type in \cite{HLL}. In this paper we are mostly interested in Grassmannians $\G(k,n)$, parametrizing linear subspaces of dimension $k$ in $\PP^n$, for which $\ed(\G(k,n))=n$.
\end{example}

Knowledge of the effective good divisibility can be used to prove the non existence of nonconstant morphisms, as exemplified by the following result.

\begin{theorem}[Cf. {\cite[Theorem 1.3]{MOS8}, \cite[Theorem 1.4]{HLL}}]\label{teo:mos8}
Let $M$ be a smooth complex projective variety, and let $M'$ be a rational homogeneous manifold of classical type, such that $\ed(M)>\ed(M')$. Then there are no nonconstant morphisms from $M$ to $M'$.
\end{theorem}

In order to prove the non existence of nonconstant morphisms between two Grassmannians $\G(l,n)$ and $\G(k,n)$, which have the same effective good divisibility we need to characterize the pairs of effective zero divisors $(x_i,x_j)$ of minimal total codimension $i+j$; they have been introduced in \cite[Definition 2.3]{MOS8}. In the definition given below we consider also the type of the pair, which keeps track of the codimensions of $x_i$ and $x_j$.

\begin{definition}\label{def:Kleiman} A (non-ordered) pair $\{x_i,x_j\}$ with $x_i \in \DA^{i}(M),$  $x_j \in \DA^{j}(M)$ nonzero effective cycles such that $x_ix_j=0$ and $i+j=\ed(M) +1$  will be called a {\em  maximal disjoint pair (\edp for short) of type $\{i,j\}$}  in $M$.
\end{definition}

We will prove Theorem \ref{thm:main} using the fact that two non-isomorphic Grassmannians $\G(l,n)$ and $\G(k,n)$ do not possess \edps of the same type, so we review in the next subsection basic facts about the generators of the Chow ring $\DA^\bullet(\G(k,n))$.

\subsection{Schubert varieties}

We will recall some basic facts about Schubert calculus in $\G(k,n)$. We refer to \cite{3264} and \cite{Fult} for the proofs and for a complete account on the subject.

Let us identify $\G(k,n)$ with the Grassmannian $G(k+1,V)$ of vector subspaces of dimension $k+1$ in a vector space $V$ of dimension $n+1$ and consider a complete flag $\cV$ of vectors subspaces of $V$:
\[0 \subsetneq V_1 \subsetneq V_2 \subsetneq \dots \subsetneq V_n \subsetneq V.\]
Given a sequence of integers $I=\{0 < i_1 < \dots < i_{k+1} \le n+1\}$, called a {\em Schubert symbol}, we define
the Schubert variety $X_I$  as
\begin{equation}
X_I = \{W \in G(k+1,V)\ |  \dim(V_{i_j} \cap W) \ge j\ \text{for all}\ j \}. \label{eq:schubert}
\end{equation}

To the Schubert variety $X_I$ one can associate a {\em Young diagram} $\la_I$ in the following way:
consider a rectangle with $k+1$ rows and $n-k$ columns and the path from the lower-left corner
to the upper-right corner consisting of $n+1$ steps, where the $i$-th step is vertical if 
$i \in I$ and horizontal if $i \not \in I$. The Young diagram is the part of
the rectangle that is top-left of this path. 

Identifying the Young diagram $\la_I$ with the {\em partition} $(\la_1, \dots, \la_{k+1})$, where $\la_i$ is the number of boxes in row $i$, we have:
\[n-k \ge \la_1 \ge\la_2 \ge \dots\ge \la_{k+1} \ge 0 \qquad \text{and} \qquad 
\la_j=i_{k+2-j} -(k+2-j).\]
By abuse, we denote by $\la_I$ also the corresponding { partition} $(\la_1, \dots, \la_{k+1})$, and we set $|I|:=|\la_I|=\sum \la_i$; this number is the dimension of the subvariety $X_I$.

 The {\em Bruhat order} on the set of Schubert symbols is defined as follows: $I \le L$ if and only if $i_j \le l_j$ for every $j=1, \dots, n+1$. Notice that $I \le L$ if and only if the Young diagram of $I$ is a subdiagram of the Young diagram of $L$.
 
The class $[X_I] \in \DA^\bullet(\G(k,n))$ does not depend on the choice of the flag $\cV$, and will be called a Schubert cycle. The Schubert cycles form a basis of $\DA^\bullet(\G(k,n))$.
 
\begin{example}\label{ex:mdp}
Let $H=\PP(V_n)$; the Schubert variety $X_H:=X_{I_H}$, corresponding to the Schubert symbol
$I_H=\{ n-k < n-k+1 < \dots < n\}$  parametrizes  $(k+1)$-dimensional subspaces of $V_n$, or equivalently $k$-dimensional linear subspaces of $\PP(V)$ contained in $H=\PP(V_n)$. The associated Young diagram  $\la_{I_H}$ is the diagram whose partition is  
$(n-k-1, \dots, n-k-1)$.

Let $p=\PP(V_1)$; the Schubert variety $X_p:=X_{I_p}$, corresponding to the Schubert symbol
 $I_p= \{ 1 <  n-k +2 < \dots < n < n+1\}$, parametrizes  $(k+1)$-dimensional subspaces of $V$ containing $V_1$,  or equivalently $k$-dimensional linear subspaces of $\PP(V)$ containing $p$. The associated Young diagram  $\la_{I_p}$ is the diagram whose partition is  
$(n-k, \dots, n-k, 0)$.
Let us show the diagrams $\la_{I_H}$ and $\la_{I_p}$ for $n=6, k=2$. \par \smallskip
\[
\la_{I_H}=
\begin{ytableau}
*(gray!40) & *(gray!40) &  *(gray!40) & \\
*(gray!40) & *(gray!40)  &  *(gray!40) & \\
*(gray!40) & *(gray!40)  &  *(gray!40) & 
\end{ytableau}
\qquad\qquad
\la_{I_p}=\begin{ytableau}
*(gray!40) & *(gray!40) &  *(gray!40) &*(gray!40) \\
*(gray!40) & *(gray!40)  &  *(gray!40) & *(gray!40)\\
 &   &  & 
\end{ytableau}
\]
\end{example}\par\medskip

A dual description of Schubert varieties, which highlights the codimension of the variety rather than the dimension, is possible. Given a Schubert symbol  $I=\{0 < i_1 < \dots < i_{k+1} \le n+1\}$ the 
 {\em dual Schubert symbol} is defined by setting $I^\vee=\{n+2-i_{k+1}  < \dots < n+2-i_{1}\}$.
The  associated Young diagram   corresponds to the partition $(\la_1^\vee, \dots, \la_{k+1}^\vee)$, where
\begin{equation}\label{eq:dual}
 \la_j^\vee=n+2 - i_j-(k+2-j)=
n-k-\la_{k+2-j}.
\end{equation}
In particular the diagram $\la_{I^\vee}$ is obtained from the diagram $\la_I$ by taking the complement and rotating it by $180^\circ$. Since $|\la_{I}|+|\la_{I^\vee}| = (k+1)(n-k) = \dim \G(k,n)$, we see that $|I^\vee|=|\la_{I^\vee}|= \codim X_{I}$.

\begin{example}\label{ex:mdpdual} Let $I_p$ and $I_H$ be the Schubert symbols introduced in Example \ref{ex:mdp}. Then the Young diagram of $\la_{I_H^\vee}$ corresponds to the partition $(1, \dots, 1)$, while the one of
$\la_{I_p^\vee}$ corresponds to the partition $(n-k, 0 , \dots, 0)$. Again, we show the diagrams for the case $n=6, k=2$.\par\smallskip
\[
\la_{I_H^\vee}=
\begin{ytableau}
*(gray!40) &  &   & \\
*(gray!40) &  &   & \\
*(gray!40) &  &   & \\
\end{ytableau}
\qquad\qquad
\la_{I_p^\vee}=\begin{ytableau}
*(gray!40) & *(gray!40) &  *(gray!40) &*(gray!40) \\
 &   &  & \\
 &   &  & 
\end{ytableau}
\]
\end{example}\par\medskip

Given a Schubert index $I$ we define the {\em opposite Schubert variety} $X^I$ to be $w_0X_{I^\vee}$, where $w_0$ is the longest element of the Weyl group of $\GL(n+1,\C)$, which acts on the canonical basis  as $w_0(\be_i)=\be_{n+1-i}$. Clearly we have an equality of cycles $[X^I]=[X_{I^\vee}] \in \DA^{|I|}(\G(k,n))$. \par\medskip

We are going to use the following well-known fact \cite[Section 1.3 and Proposition 1.3.2]{Brion} or \cite[Lemma 3.1 (3)]{Rich}:

\begin{proposition}\label{prop:ric}
The intersection of a Schubert variety $X_J$  and an opposite Schubert variety $X^I$ is nonempty
if and only if  $I \le J$. This is the case if and only if the intersection product $[X^I]\cdot[X_J]$ is nonzero.
\end{proposition}

\begin{example}
Let $I_H$ and $I_p$ be as in Example \ref{ex:mdp}. Clearly, recalling the geometric descriptions of $X_{I_H}$ and $X_{I_p}$, we have that $[X_{I_H}]\cdot [X_{I_p}]=0$, but we can obtain this also using Proposition \ref{prop:ric}, since $I_H^\vee \not \le I_p$ (and dually $I_p^\vee \not \le I_H$).
\end{example}
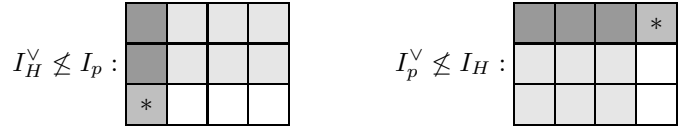
\[
I_H^\vee \not \le I_p:
\begin{ytableau}
*(gray!80) &*(gray!20) &  *(gray!20) &*(gray!20) \\
*(gray!80) &  *(gray!20) &  *(gray!20) &*(gray!20) \\
*(gray!50) * &  &   & \\
\end{ytableau}
\qquad\qquad
I_p^\vee \not \le I_H:
\begin{ytableau}
*(gray!80) & *(gray!80) &  *(gray!80) &*(gray!50) * \\
 *(gray!20) &  *(gray!20) &*(gray!20)   & \\
 *(gray!20) &  *(gray!20) &*(gray!20)   & \\
\end{ytableau}
\]

\section{Md-pairs in Grassmannians}

In this section we will show that the only \edp  for $\G(k,n)$, $1 \le k \le n-2$ is essentially the one described in Example \ref{ex:mdp}; let us start by considering \edps whose elements are Schubert cycles.

\begin{theorem}\label{thm:md} In the Grassmannian  $\G(k,n)$, with $1 \le k \le n-2$, let $X_I, X_J$ be  Schubert varieties  with $\codim(X_I)+ \codim (X_J) \le n+1$. Then $[X_I]\cdot [X_J]=0$ if and only if $\{I,J\}=\{I_H,I_p\}$.
\end{theorem}

\begin{proof} Since $[X_I]=[X^{I^\vee}]$, by Proposition \ref{prop:ric} we know that $[X_I]\cdot [X_J]\not =0$ if and only if $I^\vee \le J$. This in turn happens if and only if the Young diagram of $I^\vee$ is contained in the Young diagram  of $J$, i.e., the associated partitions $\la_{I^\vee}$ and $\mu_J$ satisfy $\la_i^\vee \le \mu_i$ for  $1 \le i \le k+1$. Note that the assumption on the codimensions can be rewritten  as 
\[|\la_{I^\vee}| + (\dim \G(k,n)- |\mu_J|) \le n+1\]
which can be restated as
\begin{equation}
|\la_{I^\vee}|  \le   |\mu_J|  - k(n-k)+(k+1).
\end{equation}
Now, from Proposition \ref{prop:comp} below we get 
$(\la_{I^\vee},\mu_J)=((n-k, 0, \dots, 0),(n-k-1, \dots, n-k-1))$ or $((1, \dots, 1),(n-k, \dots, n-k, 0))$ and from Example \ref{ex:mdp} we obtain the  statement.
\end{proof}

\begin{proposition}\label{prop:comp} Given integers $n,k$ such that $1 \le k \le n-2$, let 
 $\la$ and $\mu$ be two partitions such that
\begin{equation}\label{eq:part0}
n-k \ge \la_1 \ge \dots\ge \la_{k+1} \ge 0 \qquad \qquad n-k \ge \mu_1  \ge \dots\ge \mu_{k+1} \ge 0
\end{equation}
and
\begin{equation}\label{eq:part1}
|\la|  \le   |\mu|  - k(n-k)+(k+1).
\end{equation}
Then
$\la \not \le \mu$ if and only if 
$\la=(n-k, 0, \dots, 0)$, $\mu=(n-k-1, \dots, n-k-1)$ or $\la=(1, \dots, 1)$, $\mu=(n-k, \dots, n-k, 0)$.
\end{proposition}

\begin{proof}
In order to have  $\la \not \le \mu$ we must have   $\la_{h+1} > \mu_{h+1}$ for some $h \in \{0, \dots, k\}$. We write $\la_{h+1}=\mu_{h+1}+m$ for some positive integer $m$; let us distinguish two cases: $h >0$ or $h=0$. 
\par\medskip
\noindent \framebox{$h=0$}
\par\medskip
\noindent Using the inequalities (\ref{eq:part0}) we see that
\begin{equation}\label{eq:part01}
|\mu|-\la_1 \le k\mu_{1} -m = k\la_1-(k+1)m \le k(n-k) -(k+1)m
\end{equation}
hence, using (\ref{eq:part1}) we obtain
\begin{equation}\label{eq:part02}
0 \le \sum_{j \ge 2} \la_j = |\la| - \la_{1} \le (k+1)(1-m)
\end{equation}
forcing $m=1$ and $\la_j=0$ for $j \ge 2$. Moreover, all inequalities in (\ref{eq:part01}) and (\ref{eq:part02}) are equalities. In particular
\[
\sum_{j \ge 2} \mu_j -1 = |\mu|-\la_1 = k(n-k-1) -1
\]
so that $\mu_j=n-k-1$ for every $j \ge 2$. Now, since $\mu_2 \le \mu_1 < \la_1\le n-k$, we get  $\mu_1=n-k-1$ and $\la_1=n-k$ (Note that here we are using the assumption $k \ge 1$). We have thus proved that  $\la=(n-k, 0, \dots,0)$, $\mu=(n-k-1, \dots, n-k-1)$.\par\medskip
\noindent \framebox{$h>0$}
\par\medskip
\noindent We use inequalities (\ref{eq:part0}) to obtain
\[ 
\sum_{j >h+1} \mu_j \le (k-h)\mu_{h+1} 
\le  (k-h)(n-k-m) 
\qquad
\text{and}
\qquad
  \sum_{j \le h} \mu_j \le h(n-k).
\]
Combining the two inequalities we get 
\begin{equation}\label{eq:part1.5}
|\mu| -\la_{h+1}\le -(k-h+1)m +  k(n-k). 
\end{equation}
Now, from (\ref{eq:part1}) we obtain
\begin{equation}\label{eq:part2}
|\la| - \la_{h+1} \le -(k-h+1)m +  k+1= hm+(k+1)(1-m).
\end{equation}
On the other hand, using again (\ref{eq:part0}), we get
\begin{equation}\label{eq:part3}
|\la|-\la_{h+1} \ge  \sum_{j \le h} \la_j \ge  h\la_{h+1} \ge hm.
\end{equation}

Combining (\ref{eq:part2}) and (\ref{eq:part3}) we obtain that  $m= 1$; moreover equality holds everywhere in (\ref{eq:part1.5}), (\ref{eq:part2}) and (\ref{eq:part3}). In particular
\begin{equation} 
\la_{h+1}=1, \quad |\mu|= k(n-k) -(k-h)\label{eq:b}
\qquad\text{and} \qquad  |\la|=h+1.
\end{equation}
Since  $\la_{h+1}=1$ we have $\mu_{h+1}=0$, hence $\mu_j=0$ for $j \ge h+1$. Then 
\[|\mu| = \sum_{j \le h} \mu_j \le h(n-k),\] 
which, combined with (\ref{eq:b}), recalling that $k \le n-2$, gives $h=k$ and $\mu_j=n-k$ for every $j \le k$, so $\mu=(n-k, \dots, n-k, 0)$. By (\ref{eq:b}) we have $|\la|=k+1$. Recalling that  $\la_j \le \la_{k+1}=1$ for every $j$, we conclude that $\la=(1, \dots, 1)$. 
\end{proof}

We can now describe the \edps of $\G(k,n)$. 

\begin{corollary}\label{cor:mdp}
Let $k,n$ be integers such that  $1 \le k \le n-2$, and let $[\Gamma] \in \DA^i(\G(k,n)), [\Delta] \in \DA^j(\G(k,n))$ be effective nonzero cycles such that $[\Gamma] \cdot [\Delta]=0$ and $i+j \le n+1$. Then \{$[\Gamma],[\Delta]\}=\{a[X_{H}],b[X_{p}]\}$ with $X_H$ and $X_p$ as in Example \ref{ex:mdp}. In particular all the md-pairs  in $\G(k,n)$ have type $\{k+1,n-k\}$. 
\end{corollary}

\begin{proof}
By \cite[Corollary of Theorem 1]{FMSS}, 
the cones of effective classes of a fixed codimension in $\G(k,n)$ are polyhedral cones generated
by the Schubert classes of the same codimension. Therefore
we can write $[\Gamma]$ and $[\Delta]$ as linear combinations  with nonnegative coefficients:
\[[\Gamma]=\sum_{|K^\vee|=i} \gamma_K[X_K] \qquad  [\Delta]=\sum_{|L^\vee|=j} \delta_L[X_L];\]
Moreover every product $[X_K]\cdot [X_L]$ is a combination  of Schubert cycles with nonnegative coefficients, due to the Littlewood--Richardson rule. Then, if $\gamma_K\delta_L \not=0$ we must have $[X_K]\cdot [X_L]=0$, and the statement now follows from Theorem \ref{thm:md}.
\end{proof}

\section{Morphisms to Grassmannians}

In this section we will prove Theorem \ref{thm:main}. We state and prove first two auxiliary results that could be useful to study morphisms to Grassmannians from other kinds of varieties.

\begin{lemma}\label{lem:red}
Let $k,n$ be integers such that  $1 \le k\le n-2$, and let $Y \subset \G(k,n)$ be a positive dimensional closed irreducible subvariety. Then
\begin{itemize}[leftmargin=15pt]
\item If $[Y] \cdot [X_H]=0$  there is a nonconstant morphism $\psi:Y \to \G(k-1,n-1)$.
\item If $[Y] \cdot [X_p]=0$  there is a nonconstant morphism $\psi:Y \to \G(k,n-1)$.
\end{itemize}
\end{lemma}

\begin{proof}
In the first case, for a general hyperplane $H'$, $Y$ does not meet the subvariety $X_{H'} \subset \G(k,n)$ parametrizing  linear spaces contained in $H'$. We can take $\psi$ to be the restriction to $Y$ of the morphism $\pi_H:\G(k,n) \setminus X_{H'} \to \G(k-1,n-1)$, which sends $\Lambda$ to $\Lambda \cap H'$. Since the fibers of this morphism are affine (see \cite[Example 5.5]{MOS8}), we get that $\psi$ is not constant.

The argument in the second case is similar: for a general point $q \in \PP^n$,  $Y$ does not meet the subvariety $X_q \subset \G(k,n)$ parametrizing linear spaces passing through $q$.  We can take $\psi$ to be the restriction to $Y$ of the morphism $\pi_q:\G(k,n) \setminus X_q \to \G(k,n-1)$, which sends $\Lambda$ to the linear projection of $\Lambda$ from $q$ to a hyperplane. Again the fibers of this morphism are affine and $\psi$ is not constant.
\end{proof}

\begin{proposition}\label{prop:gen} Let $k,n$ be integers such that $1 \le k\le n-2$ and let 
$M$ be a smooth variety with $\ed(M)=n$ in which there is no \edp of type $\{k+1,n-k\}$.  
Then every morphism $\f:M \to \G(k,n)$ is constant.
\end{proposition}

\begin{proof}
Let $[X_H],[X_p] \in \DA^\bullet(\G(k,n))$ be as in Example \ref{ex:mdp}.  
For general $g,g' \in \GL(n+1, \C)$, $gX_H$ and $g'X_p$  are disjoint and generically transverse to $\varphi$, that is, $\varphi^{-1}(gX_H), \varphi^{-1}(g'X_p)$ are generically reduced and of the same codimensions as $X_H,X_p\subset X$ (see \cite[Theorem 1.7]{3264}). 

By \cite[Theorem 1.23]{3264}, we have that  $\varphi^*[X_H]=[\varphi^{-1}(gX_H)] \in \DA^{k+1}(M)$ and $\varphi^*[X_p]=[\varphi^{-1}(g'X_p)]  \in \DA^{n-k}(M)$. In particular 
\[\codim \varphi^*[X_H] + \codim \varphi^*[X_p] = n+1 = \ed(M)+1\] 
and 
$\varphi^*[X_H]\cdot \varphi^*[X_p]=0$. By assumption $(\varphi^*[X_H],\varphi^*[X_p])$ cannot be an \edp for $M$, so one of the two cycles is zero.

If $\f$ were not constant, by Lemma \ref{lem:red} we would get a nonconstant morphism $\psi\circ \f$ from $M$ to  $\G(k-1,n-1)$ or $\G(k,n-1)$, contradicting \cite[Proposition 2.4]{NO22}.
\end{proof}

\begin{proof}[Proof of Theorem \ref{thm:main}] Since the Picard number of  $\G(l,n)$ is one, a nonconstant morphism must be finite. In fact, if this were not the case, the pullback of an ample line bundle on the image will be ample on $\G(l,n)$, but trivial on the contracted curves, contradicting Kleiman's criterion.   If $k=0,n-1$  we then have a morphism $\f: \G(l,n) \to \PP^n$ which is constant since $\dim  \G(l,n) >n$. Else, by Proposition \ref{prop:gen}, if $\f$ is not constant then $l=k$ or $n-k-1$. In this case, since  the dimensions of the domain and the codomain are equal, $\f$ must be surjective. We conclude that $\f$ is an isomorphism by \cite[Main Theorem]{HM}.
\end{proof}

\subsection*{Acknowledgements}

We would like to thank an anonymous referee for useful comments and suggestions.

\subsection*{Data availability statement}
Data sharing not applicable to this article as no datasets were generated or analyzed during the current study.

\subsection*{Conflict of interests}
The authors have no relevant financial or non-financial interests to disclose.

\bibliographystyle{plain}

\end{document}